\documentclass[smallcondensed]{svjour3} 
\usepackage{geometry}
\usepackage[english]{babel}
\usepackage{float}
\usepackage{amsmath}
\usepackage{graphicx}
\usepackage[numbers]{natbib}
\usepackage{amssymb}
\makeatletter
\providecommand{\tabularnewline}{\\}

\numberwithin{equation}{section}
\numberwithin{figure}{section}
\usepackage{amsfonts}
\usepackage{graphicx}
 \usepackage{layout}
\makeatother

\begin{document}

\title{Bernstein dual-Petrov-Galerkin method: application to 2D time fractional
diffusion equation}

\author{M. Jani \and S. Javadi \and E. Babolian \and D. Bhatta}

\institute{M. Jani \and S. Javadi \and  E. Babolian \at Department of Mathematics,
Faculty of Mathematical Sciences and Computer, Kharazmi University,
Tehran, Iran\\
\email{mostafa.jani@gmail.com javadi@khu.ac.ir babolian@khu.ac.ir.}
\\
\and D. Bhatta \at School of Mathematical \& Statistical Sciences,
The University of Texas Rio Grande Valley, 1201 West University Drive,
Edinburg, TX, USA\\
\email{dambaru.bhatta@utrgv.edu} \\
}
\maketitle
\begin{abstract}
In this paper, we develop a Bernstein dual-Petrov-Galerkin method
for the numerical simulation of a two-dimensional fractional diffusion
equation. A spectral discretization is applied by introducing suitable
combinations of dual Bernstein polynomials as the test functions and
the Bernstein polynomials as the trial ones. We derive the exact sparse
operational matrix of differentiation for the dual Bernstein basis
which provides a matrix based approach for the spatial discretization.
It is shown that the method leads to banded linear systems that can
be solved efficiently. The stability and convergence of the proposed
method is discussed. Finally, some numerical examples are provided
to support the theoretical claims and to show the accuracy and efficiency
of the method.

\end{abstract}

\section{Introduction}

Bernstein polynomial basis plays an important role in computer graphics
for geometric modeling, curve and surface approximation. Some interesting
features have been investigated for this basis in the last decades;
for instance, it is proven to be an optimal stable basis among nonnegative
bases in a sense described in \citep{farouki2-1996}. Also, it provides
some optimal shape preserving features \citep{carnicer}. We refer
to \citep{farin2002,farouki1988,Farouki1999} for detailed properties
and applications in computer aided geometric design (CAGD). 

Bernstein basis has also been used for the numerical solution of differential,
integral, integro-differential and fractional differential equations,
see e.g. \citep{Behiry,javadi,JavadiJaniBabolian,maleknejad,saadatmandi}
and the references therein. However, it is not orthogonal and so leads
to dense linear systems when using in numerical methods. Some numerical
approaches implement the orthogonalized Bernstein basis. However,
as we will see in the next section, it fails to keep some interesting
properties of the Bernstein basis. Another approach uses the dual
Bernstein polynomials (DBP) introduced by Juttler in 1998 \citep{juttler}.
To the best of our knowledge, the DBP basis has been only discussed
from the CAGD point of view (see the works of Lewanowicz and Wozny
e.g. \citep{Lewanowicz,wozny}). So it is of interest to explore some
new aspects of this basis in order to facilitate the numerical methods
for differential equations that are based on Bernstein polynomials
and to present a method for time fractional diffusion equation in
two dimensions.

Fractional partial differential equations (FPDEs) have been widely
used for the description of some important physical phenomena in many
applied fields including viscoelastic materials, control systems,
polymer, electrical circuits, continuum and statistical mechanics,
etc., see, for instance \citep{Dabiri,Goychuk,metzler2004,Moghaddam,YangMachado}
and the references therein. The subdiffusion equation is a FPDE describing
the behavior of anomalous diffusive systems with the probability density
of particles diffusing proportional to the mean square displacement
$\chi^{2}(t)\propto t^{\alpha}$ with $0<\alpha<1$ \citep{Gao2012FDM}.
Anomalous diffusion equations have been used for modeling transport
dynamics, especially the continuous time random walk, the contamination
in porous media, viscoelastic diffusion, etc \citep{Gao2012FDM,Gao2015,Goychuk,metzler2004,Wang2016}.
For the numerical solution of the one-dimensional problem, we refer
to \citep{JinLazaroV2016,Ren,Stok,zhou} and the references therein.
Some classic numerical methods for PDEs have been developed for the
simulation of two-dimensional subdiffusion equation, for example the
finite difference schemes \citep{Gao2015,pang,Ren}, meshless methods
\citep{Shirzadi,Yang}, finite element method \citep{ZhaoFEM}, alternating
direction implicit methods \citep{ZhangADI,ZhangCompactADI}, etc.

In this paper, deriving some new aspects of DBPs, we present suitable
combinations of these functions in order to develop a dual-Petrov-Galerkin
method for solving the following 2D subdiffusion equation \citep{Wang2016,Yang,Yang2014,ZhangADI,ZhangCompactADI,ZhaoFEM}
\begin{equation}
{D_{t}^{\alpha}}u\left(x,y,t\right)=\kappa\Delta u\left(x,y,t\right)+S\left(x,y,t\right),\quad\left(x,y,t\right)\in\Omega\times\left(0,T\right],\label{eq:main}
\end{equation}
with the following initial and boundary conditions 
\begin{align}
 &u\left(x,y,0\right)=g\left(x,y\right),\quad(x,y)\in\Omega,\label{IV}\\
 &u\left(x,y,t\right)=0,\quad\left(x,y,t\right)\in\partial\Omega\times\left(0,T\right],\label{BVs}
\end{align}
where $\Omega=\left(0,1\right)^{2}\subset\mathbb{R}^{2}$, $\Delta$
is the Laplacian operator, $T>0$, $\kappa$ is the diffusion coefficient
and $S$ is the source function. Here, ${D_{t}^{\alpha}}u$ denotes
the Caputo fractional derivative of order $\alpha,$ $0<\alpha<1$,
with respect to $t$ defined as
\begin{equation}
 {D_{t}^{\alpha}}u\left(x,t\right)=\frac{1}{\Gamma\left(1-\alpha\right)}\int_{0}^{t}{\frac{1}{\left(t-s\right){}^{\alpha}}\frac{\partial u\left(x,s\right)}{\partial s}ds},\quad0<\alpha<1.\label{eq:fracDef}
\end{equation}

The main contribution of our work is the development of an accurate
Bernstein dual-Petrov-Galerkin method and the application for the
numerical simulation of the 2D subdiffusion equation. It is shown
the method leads to sparse linear systems. To give a matrix approach
of the method, we present some results concerning the DBPs including
a recurrence formula for the derivative, constructing a new basis
using DBPs, deriving the operational matrix of differentiation and
also providing the transformation matrices between the DBPs and the
new basis.

The paper is organized as follows: Section \ref{sec:Bern=000026Dual}
presents some new aspects of DBPs and provides modal basis functions
and the associated transformation matrices between the bases. Section
\ref{sec:variationalFormulation} is devoted to the Bernstein-spectral
formulation of the subdiffusion problem (\ref{eq:main})-(\ref{BVs})
and the stability and convergence results are discussed in Section
\ref{sec:Error-estimation}. Numerical examples are provided in Section
\ref{sec:Numerical-examples}. The paper ends with some concluding
remarks in Section \ref{sec:Con}.

\section{\label{sec:Bern=000026Dual}Bernstein polynomials and DBPs}

The Bernstein polynomials with degree $N$ on the unit interval are
defined by 
\begin{equation*}
\phi_{i}(x)=\dbinom{N}{i}x^{i}\left(1-x\right){}^{N-i},\qquad0\leq i\leq N.
\end{equation*}
The set $\left\{ \phi_{i}\left(x\right):\,i=0,\dots,N\right\} $ forms
a basis for $\mathbb{P}_{N}$, the space of polynomials of degree
not exceeding $N$\@.

These polynomials possess end-point interpolation property, i.e.,
\begin{equation}
 \phi_{i}\left(0\right)=\delta_{i,0},\quad\phi_{i}\left(1\right)=\delta_{i,N},\quad0\leq i\leq N,\,N>0.\label{eq:BernBound}
\end{equation}
Also, the summation is one and the integral over the unit interval
is constant, namely 
\begin{equation}
 \sum_{i=0}^{N}{\phi_{i}(x)}\equiv 1,\quad\int_{0}^{1}{\phi_{i}(x)}=\frac{1}{N+1},\,i=0,1,\dots,N.\label{eq:sumAndIntOfBern}
\end{equation}
The derivative enjoys the three-term recurrence relation \citep{janiNA}
\begin{equation}
 \phi_{i}^{\prime}\left(x\right)=\left(N-i+1\right)\phi_{i-1}\left(x\right)-\left(N-2i\right)\phi_{i}\left(x\right)-\left(i+1\right)\phi_{i+1}\left(x\right),\quad0\leq i\leq N,\label{eq:BernDeriv}
\end{equation}
where we adopt the convention that $\phi_{i}\left(x\right)\equiv0$
for $i<0$ and $i>N$\@. 

As we mentioned in the preceding section, the Bernstein basis is not
orthogonal. The corresponding orthogonalized basis, obtained e.g.,
by the Gram-Schmidt process fails to keep some interesting aspects
of the original basis. We will not consider this basis in the present
work. Instead we turn to the dual basis.

The DBPs are defined as
\begin{equation}
\tilde{\psi}_{i}\left(x\right)=\sum_{j=0}^{N}c_{i,j}\phi_{j}\left(x\right),\label{eq:Dual}
\end{equation}
with the coefficients given by 
\begin{equation}
c_{i,j} =  \frac{(-1)^{i+j}}{\binom{N}{i}\binom{N}{j}}\sum_{r=0}^{\min\left(i,j\right)}\left(2r+1\right)\binom{N+r+1}{N-i}\binom{N-r}{N-i}\binom{N+r+1}{N-j}\binom{N-r}{N-j}.\label{eq:DualCoefficients}
\end{equation}
It is verified that they satisfying the biorthogonal system {[}\citealp{juttler},
Theorem 3{]}
\begin{equation}
\int_{0}^{1}\phi_{i}\left(x\right)\tilde{\psi}_{j}\left(x\right)dx = \delta_{ij},\quad0\leq i,j\leq N.\label{eqBiorthogonality}
\end{equation}
It is worth noting that less than a quarter of the entries of transformation
matrix between the Bernstein and dual Bernstein basis $C=[c_{i,j}]$,
are to be computed directly by (\ref{eq:DualCoefficients}); for it
is bisymmetric, i.e., symmetric about both of the main diagonal and
antidiagonal. 

Another property which is used later is that the sum of the entries
for each row (column) is equal to the order of the matrix, i.e.,
\begin{equation}
 \sum_{i=0}^{N}{c_{i,j}}=\sum_{j=0}^{N}{c_{i,j}}=N+1.\label{eq:SumOfRowIsConstant}
\end{equation}
In the next lemma, we present some properties of the DBPs.
\begin{lemma}
\label{Dual properties}Let $N$ be a nonnegative integer. The following
statements hold.

(i) For all $x\in[0,1]$, $\tilde{\psi}_{N-i}\left(x\right)=\tilde{\psi}_{i}\left(1-x\right)$,
$0\leq i\leq N$\@.

(ii) For all $x\in[0,1]$, $\sum_{i=0}^{N}{\tilde{\psi}_{i}\left(x\right)}=N+1.$

(iii) The basis functions have the same definite integral, i.e.,$\int_{0}^{1}\tilde{\psi}_{i}\left(x\right)dx=1,\quad0\leq i\leq N$.
\end{lemma}
\begin{proof}
The first statement is an immediate consequence of the similar formula
for Bernstein polynomials, i.e., $\phi_{N-i}(x)=\phi_{i}(1-x).$ From
(\ref{eq:Dual}), (\ref{eq:SumOfRowIsConstant}) and (\ref{eq:sumAndIntOfBern}),
we have 
\begin{align*}
\sum_{i=0}^{N}{\tilde{\psi}_{i}\left(x\right)} & =  \sum_{i=0}^{N}{\sum_{j=0}^{N}c_{i,j}\phi_{j}\left(x\right)}\\
 & =  \sum_{j=0}^{N}{\phi_{j}\left(x\right)\sum_{i=0}^{N}c_{i,j}}=N+1.
\end{align*}
statement (iii) is also verified similarly.
\end{proof}

The property (i) implies that $\tilde{\psi}_{i}$, for $\left[\frac{N}{2}\right]+1\leq i\leq N,$
need not to be computed directly by (\ref{eq:Dual})-(\ref{eq:DualCoefficients}).
It especially gives $\tilde{\psi}_{i}\left(0\right)=\tilde{\psi}_{N-i}\left(1\right)$.

\subsection{Modal basis functions}

One may choose a suitable compact combinations of orthogonal polynomials
as the trial and test basis for the Galerkin and Petrov-Galerkin methods
for BVPs in such a way leading to sparse linear systems for some special
problems (see e.g., \citep{GoubetShenDual,YuanShenDual}). Here, we
use this idea for the non-orthogonal Bernstein polynomials to present
a simple and accurate dual-Petrov-Galerkin spectral method for two-dimensional
subdiffusion equation. Following Shen, et. al. {[}\citealp{GoubetShenDual}
and \citealp{Shen}, Section 1.3{]}, we will refer to such basis functions
as the modal basis functions.
\begin{proposition}
\label{Prop 1. vanishPsi}Let $N\geq2$ be an integer, $\{\tilde{\psi}_{i}:\,0\leq i\leq N\}$
be the dual Bernstein basis and $\mathbb{P}_{N}^{0}=\{u\in\mathbb{P}_{N}:\,u(0)=0,u(1)=0\}.$
Set 
\begin{equation}
\psi_{i}\left(x\right)  =  \tilde{\psi}_{i}\left(x\right)+a_{i}\tilde{\psi}_{i+1}\left(x\right)+b_{i}\tilde{\psi}_{i+2}\left(x\right),\label{eq:CompactBasis}
\end{equation}
for $0\leq i\leq N-2,$ where 
\begin{align}
a_{i} & =  \frac{2i+4}{N-i+1},\nonumber \\
b_{i} & =  \frac{(i+2)(i+3)}{(N-i)(N-i+1)}.\label{eq:a=00005Bi=00005D and b=00005Bi=00005D}
\end{align}
Then, the polynomials $\tilde{\psi}_{i}(x)$ vanish at 0 and 1, so
the set $\{\psi_{i}\left(x\right)\}_{i=0}^{N-2}$ forms a basis for
$\mathbb{P}_{N}^{0}$\@.
\end{proposition}
\begin{proof}
By (\ref{eq:Dual}) and (\ref{eq:BernBound}), we have
\begin{equation}
\tilde{\psi}_{i}\left(0\right)  =\sum_{j=0}^{N}c_{i,j}\phi_{j}\left(0\right)=  c_{i,0}=\left(-1\right)^{i}\left(N+1\right)\binom{N+1}{i+1}.\label{eq:DualValue0}
\end{equation}
From Lemma \ref{Dual properties}, we infer 
\begin{equation}
\tilde{\psi}_{i}\left(1\right)=\tilde{\psi}_{N-i}\left(0\right)=\left(-1\right)^{N-i}\left(N+1\right)\binom{N+1}{i}.\label{eq:DualValue1}
\end{equation}
By (\ref{eq:DualValue0}) and (\ref{eq:DualValue1}), we obtain
\begin{align*}
 & \psi_{i}\left(0\right)=\left(-1\right)^{i}\left(N+1\right)\left(\binom{N+1}{i+1}-a_{i}\binom{N+1}{i+2}+b_{i}\binom{N+1}{i+3}\right)=0,\\
 & \psi_{i}\left(1\right)=\left(-1\right)^{N-i}\left(N+1\right)\left(\binom{N+1}{N-i+1}-a_{i}\binom{N+1}{N-i}+b_{i}\binom{N+1}{N-i-1}\right)=0,
\end{align*}
for $0\leq i\leq N-2.$ It is easy to see $\{\psi_{i}\left(x\right)\}_{i=0}^{N-2}$
is linearly independent. Since $\mathrm{dim}\mathbb{P}_{N}^{0}=N-1$,
this set is a basis for $\mathbb{P}_{N}^{0}$. This completes the
proof\@.
\end{proof}

Figure \ref{fig:Dual-Bern}. illustrates the DBPs and the modal basis
functions for $6\leq N\leq8$\@. It is seen that the modal basis
functions have less values than the corresponding dual functions on
the unit interval, expecting less round-off errors\@.

\begin{figure}
\centering
\includegraphics[scale=0.35]{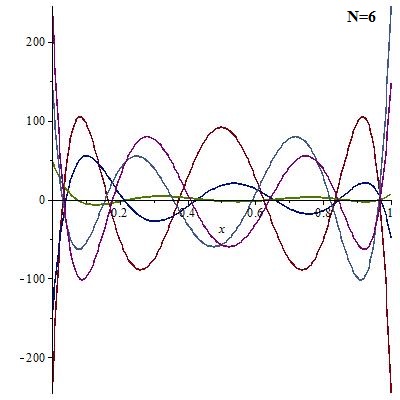}\includegraphics[scale=0.35]{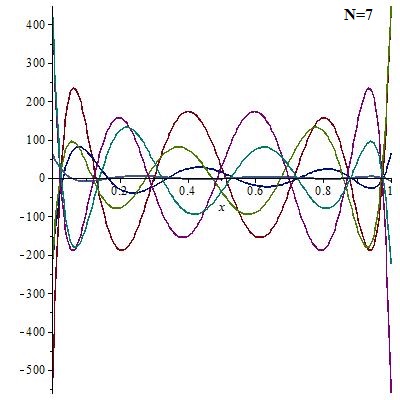}\includegraphics[scale=0.35]{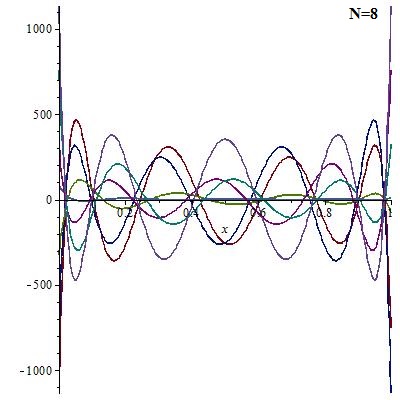}
\par
\centering
\includegraphics[scale=0.35]{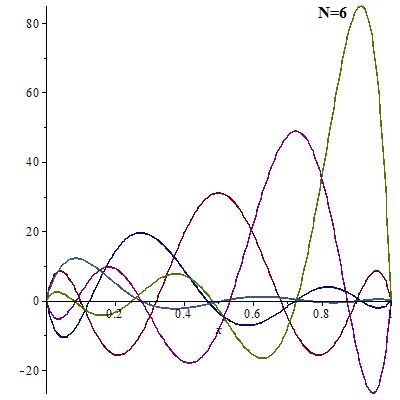}\includegraphics[scale=0.35]{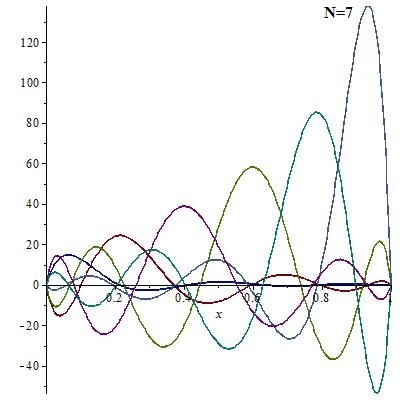}\includegraphics[scale=0.35]{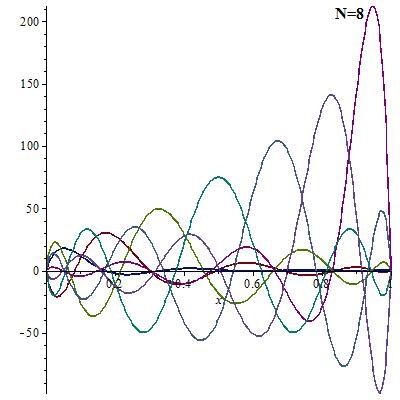}
\par
\centering{}\caption{\label{fig:Dual-Bern}Graphs of DBPs $\{\tilde{\psi}_{i}(x),\,0\leq i\leq N\}$
(top) and the modal basis functions $\{\psi_{i}(x),\,0\leq i\leq N-2\}$
(bottom).}
\end{figure}

\subsection{\label{subsec:Transformation-matrices-and}Transformation matrices
and the operational matrix for derivatives}

For $N\geq2$, consider the $(N+1)-$vector $\tilde{\mathbf{\Psi}}$
and the $(N-1)-$vector $\mathbf{\Psi}$ consisting of dual functions
given by (\ref{eq:Dual}) and the modal basis functions given by (\ref{eq:CompactBasis}),
respectively:
\begin{align}
 &   \tilde{\mathbf{\Psi}}(\cdot)=[\tilde{\psi}_{i}\left(\cdot\right):\,0\leq i\leq N]^{T},\label{eq:DualBasisVector}\\
 &   \mathbf{\Psi}(\cdot)=[\psi_{i}\left(\cdot\right):\,0\leq i\leq N-2]^{T}.\label{eq:ModalBasisVector}
\end{align}
For simplicity, we ignore the dependence of the vectors on variable.
First, note that 
\begin{equation}
 \mathbf{\Psi}=\mathbf{G}\tilde{\mathbf{\Psi}},\label{eq:SaiToSaiTilde(G)}
\end{equation}
where $\mathbf{G}=[g_{i,j}]$ is an $\left(N-1\right)\times\left(N+1\right)$
matrix with three diagonals as 
\begin{align*}
g_{i,j} & =  \begin{cases}
1, & i-j=0,\\
a_{i}, & j=i+1,\\
b_{i}, & j=i+2,
\end{cases}\quad0\leq i\leq N-1,0\leq j\leq N.
\end{align*}
To derive a formula for the derivative of the modal basis functions,
we first prove the following result.
\begin{lemma}
The operational matrix for derivative of the DBPs, $\mathbf{P}$ satisfies
\begin{equation}
\tilde{\mathbf{\Psi}}^{\prime}  =  \mathbf{P}\tilde{\mathbf{\Psi}},\label{eq:P SaiT' to SaiT}
\end{equation}
where the matrix $\mathbf{P}=[p_{i,j}:\,0\leq i,j\leq N]$ is given
by
\begin{align*}
p_{i,j}=\begin{cases}
-(-1)^{i}(N+1)\binom{N+1}{i+1}+N\delta_{i,0}+\delta_{i,1}, & j=0,\\
-p_{N-i,0} & j=N,\\
i, & j=i-1,\,j\neq0,\\
N-2i, & j=i,\,j\neq0,N\\
-N+i, & j=i+1,\,j\neq N.
\end{cases}
\end{align*}
\end{lemma}
\begin{proof}
The DBPs $\tilde{\mathbf{\Psi}}$ is a basis for $P_{N}$, so we expand
$\tilde{\psi}_{i}^{\prime}\left(x\right)$ for $0\leq i\leq N,$ as
\begin{equation*}
\tilde{\psi}_{i}^{\prime}\left(x\right)  =  \sum_{j=0}^{N}{p_{i,j}\tilde{\psi}_{j}\left(x\right)}.
\end{equation*}
Integration by parts and (\ref{eq:BernDeriv}) imply that 
\begin{align*}
p_{i,j} & =  \int_{0}^{1}{\tilde{\psi}_{i}^{\prime}\left(x\right)\phi_{j}\left(x\right)dx}\\
 & =  \tilde{\psi}_{i}\left(1\right)\delta_{j,N}-\tilde{\psi}_{i}\left(0\right)\delta_{j,0}-\int_{0}^{1}{\tilde{\psi}_{i}\left(x\right)\left((N-j+1)\phi_{j-1}(x)-(N-2j)\phi_{j}(x)-(j+1)\phi_{j+1}(x)\right)dx}.
\end{align*}
The biorthogonality (\ref{eqBiorthogonality}) gives
\begin{equation*}
p_{i,j}  =  \tilde{\psi}_{i}\left(1\right)\delta_{j,N}-\tilde{\psi}_{i}\left(0\right)\delta_{j,0}-\left((N-j+1)\delta_{i,j-1}-(N-2j)\delta_{i,j}-(j+1)\delta_{i,j+1}\right).
\end{equation*}
Now, the result is proved by considering (\ref{eq:DualValue0}) and
(\ref{eq:DualValue1})\@.
\end{proof}

\begin{remark}
The matrix $\mathbf{P}$ is a sparse matrix of order $N+1$ with $p_{i,j}=0$
for $|i-j|>1,$ $j\neq0,N$; for instance, see the matrix given below.
\end{remark}
\begin{corollary}
Set $\alpha_{i,0}=-(-1)^{i}(N+1)\binom{N+1}{i+1}+N\delta_{i,0}+\delta_{i,1}$
for $0\leq i\leq N.$ Then, from (\ref{eq:P SaiT' to SaiT}), we infer
the following five-term recurrence relation is deduced
\begin{align*}
\tilde{\psi}_{i}^{\prime}(x)  &=  \alpha_{i,0}\tilde{\psi}_{0}\left(x\right)+(1-\delta_{i,1})i\tilde{\psi}_{i-1}\left(x\right)+(1-\delta_{i,0})(1-\delta_{i,N})\left(N-2i\right)\tilde{\psi}_{i}\left(x\right)\\
 & \hfill-(1-\delta_{i,N-1})\left(N-i\right)\tilde{\psi}_{i+1}\left(x\right)-\alpha_{N-i,0}\tilde{\psi}_{N}\left(x\right),
\end{align*}
where we set $\tilde{\psi}_{i}\equiv0$ for $i<0$ and $i>N.$
\end{corollary}
We derive the transformation matrices that map the Bernstein and Chebyshev
coefficients

Now we derive the transformation matrix that maps the derivative of
modal basis functions to DBPs. This facilitates the use of Galerkin
method in the next section. In the following, $(p,q)-band$ matrix
stands for a matrix with $p$ and $q$ nonzero diagonals below and
above the main diagonal, respectively. 
\begin{lemma}
\label{Lemma1forQ}Let the vectors $\mathbf{\Psi}$ and $\tilde{\mathbf{\Psi}}$
be defined as in (\ref{eq:DualBasisVector}) and (\ref{eq:ModalBasisVector}),
respectively. Then, 
\begin{equation*}
  \mathbf{\Psi}^\prime=\mathbf{Q}\tilde{\mathbf{\Psi}},
\end{equation*}
where $\mathbf{Q}$ is an $\left(N-1\right)\times\left(N+1\right),$
$\mathrm{(1,3)-band}$ matrix given by $\mathbf{Q}=\mathbf{GP}.$
\end{lemma}
\begin{proof}
Combining (\ref{eq:SaiToSaiTilde(G)}) with (\ref{eq:P SaiT' to SaiT}),
implies $\mathbf{Q}=\mathbf{GP}.$ To prove that $\mathbf{Q}$ is
a $\mathrm{(1,3)-band}$ matrix, it is sufficient to show that $q_{i,0}=0$
for $i>1$ and $q_{i,N}=0$ for $i<N-2,$ 
\begin{align*}
q_{i,0} & =  \left(\mathbf{GP}\right)_{i,0}=p_{i,0}+a_{i}p_{i+1,0}+b_{i}p_{i+2,0}\\
 & =  -\left(\tilde{\psi}_{i}\left(0\right)+a_{i}\tilde{\psi}_{i+1}\left(0\right)+b_{i}\tilde{\psi}_{i+2}\left(0\right)\right)=-\psi_{i}(0)=0,
\end{align*}
and for $i<N-2,$ by (\ref{Dual properties})
\begin{align*}
q_{i,N} & =  p_{i,N}+a_{i}p_{i+1,N}+b_{i}p_{i+2,N}=\tilde{\psi}_{N-i}\left(0\right)+a_{i}\tilde{\psi}_{N-i-1}\left(0\right)+b_{i}\tilde{\psi}_{N-i-2}\left(0\right)\\
 & =  \tilde{\psi}_{i}\left(1\right)+a_{i}\tilde{\psi}_{i+1}\left(1\right)+b_{i}\tilde{\psi}_{i+2}\left(1\right)=\psi_{i}(1)=0.
\end{align*}
Note that $\psi_{i}$'s vanish at the boundary values according to
Proposition \ref{Prop 1. vanishPsi}. The proof is complete.
\end{proof}

To see the sparsity of the transformation matrices, $\mathbf{P}$,
$\mathbf{G}$ and $\mathbf{Q}$ for $N=6$ are shown in the following.

\begin{table}[H]
\noindent\resizebox{\textwidth}{!}{%

\begin{tabular}{ccc}
$\mathbf{G}=\left[\begin{array}{ccccccc}
1 & \frac{4}{7} & \frac{1}{7} & 0 & 0 & 0 & 0\\
0 & 1 & 1 & \frac{2}{5} & 0 & 0 & 0\\
0 & 0 & 1 & \frac{8}{5} & 1 & 0 & 0\\
0 & 0 & 0 & 1 & \frac{5}{2} & \frac{5}{2} & 0\\
0 & 0 & 0 & 0 & 1 & 4 & 7
\end{array}\right],$ & $\mathbf{P}=\left[\begin{array}{ccccccc}
-43 & -6 & 0 & 0 & 0 & 0 & 7\\
148 & 4 & -5 & 0 & 0 & 0 & -49\\
-245 & 2 & 2 & -4 & 0 & 0 & 147\\
245 & 0 & 3 & 0 & -3 & 0 & -245\\
-147 & 0 & 0 & 4 & -2 & -2 & 245\\
49 & 0 & 0 & 0 & 5 & -4 & -148\\
-7 & 0 & 0 & 0 & 0 & 6 & 43
\end{array}\right],$ & $\mathbf{Q}=\left[\begin{array}{ccccccc}
\frac{46}{7} & -\frac{24}{7} & -\frac{18}{7} & -\frac{4}{7} & 0 & 0 & 0\\
1 & 6 & -\frac{9}{5} & -4 & -\frac{6}{5} & 0 & 0\\
0 & 2 & \frac{34}{5} & 0 & -\frac{34}{5} & -2 & 0\\
0 & 0 & 3 & 10 & \frac{9}{2} & -15 & -\frac{5}{2}\\
0 & 0 & 0 & 4 & 18 & 24 & -46
\end{array}\right].$\tabularnewline
\end{tabular}

}
\end{table}

\section{\label{sec:variationalFormulation}Variational formulation of the
problem (\ref{eq:main}) and the spectral discretization}

In this section, at first the problem (\ref{eq:main})-(\ref{BVs})
is discretized in time. Then we develop a matrix approach Bernstein
dual-Petrov-Galerkin method using the results of the previous section. 

\subsection{Time discretization}

Consider the subdiffusion equation (\ref{eq:main}) at $t=t_{k+1},\,k\geq0$
as
\begin{equation}
{D_{t}^{\alpha}}u\left(x,y,t_{k+1}\right)=\kappa\Delta u\left(x,y,t_{k+1}\right)+S\left(x,y,t_{k+1}\right).\label{eq:AtTimek+1}
\end{equation}
Let $u^{k}$ be an approximation of $u$ at $t=t_{k}=k\tau$ for $k=0,1,..,M,$
where $\tau=\frac{T}{M}$ is the time step length\@. The time fractional
derivative can be approximated by definition (\ref{eq:fracDef}) and
using forward difference for the derivative inside as
\begin{equation}
{D_{t}^{\alpha}}u\left(x,y,t_{k+1}\right)=\mu(u^{k+1}-(1-b_{1})u^{k}-\sum_{j=1}^{k-1}(b_{j}-b_{j+1})u^{k-j}-b_{k}u^{0})+r_{\tau}^{k+1},\quad k\geq1,\label{eq:4.1-1}
\end{equation}
where $\mu=\frac{1}{\tau^{\alpha}\Gamma\left(2-\alpha\right)}$ and
$b_{j}=\left(j+1\right)^{1-\alpha}-j^{1-\alpha}$ for $k\geq0$ and
$0\leq j\leq k$\@. The error is bounded by 
\begin{equation}
\left|r_{\tau}^{k+1}\right|  \leq  \tilde{c}_{u}\tau^{2-\alpha},\label{eq:rateInTime}
\end{equation}
where the coefficient $\tilde{c}_{u}$ depends only on $u$ \citep{Deng}\@.
The time discretization (\ref{eq:4.1-1}) is referred to as L1 approximation
(see e.g. \citep{Deng,Rame})\@. Substituting from (\ref{eq:4.1-1})
into (\ref{eq:AtTimek+1}) and multiplying both sides by $\tau^{\alpha}\Gamma\left(2-\alpha\right)$
and dropping $(x,y)$, the following time-discrete scheme is obtained
\begin{align}
 &   u^{k+1}-\alpha_{0}\Delta u^{k+1}=f^{k+1},\quad k\geq0,\nonumber \\
 &   f^{k+1}:=(1-b_{1})u^{k}+\sum_{j=1}^{k-1}(b_{j}-b_{j+1})u^{k-j}+b_{k}u^{0}+\frac{1}{\mu}S^{k+1},\label{eq: semi-discrete}
\end{align}
with $\alpha_{0}=\frac{k}{\mu}$ and $u^{0}=g$ is given by the initial
condition (\ref{IV}) with the error 
\begin{equation}
  r^{k+1}\leq\tau^{\alpha}\Gamma\left(2-\alpha\right)|r_{\tau}^{k+1}|\leq\tilde{c}_{u}\tau^{2}.\label{eq:TruncationErrorMulti}
\end{equation}

For $k=0,$ it reads as
\begin{equation}
  u^{1}-\alpha_{0}\kappa\Delta u^{1}=(1-b_{1})u^{1}+b_{1}u^{0}+\frac{1}{\mu}S^{1}.\label{eq:semi-discFork=00003D0}
\end{equation}
The boundary conditions for the semidiscrete problem is $u^{k+1}=0$
on $\partial\Omega$. 

\subsection{Weak and spectral formulation}

Consider the problem (\ref{eq: semi-discrete}) with $\Omega=I{}^{2},\,I=\left(0,1\right)$
and the homogeneous Dirichlet boundary conditions $u^{k+1}|_{\partial\Omega}=0$\@.
We seek an approximate solution in the Sobolev space $H_{0}^{1}\left(\Omega\right)=\{u\in H^{1}(\Omega),u=0,\,\text{on \ensuremath{\partial\Omega}\}}$\@.
A weak formulation of the problem (\ref{eq: semi-discrete}) is to
find $u^{k+1}\in H_{0}^{1}(\Omega)$ such that $\forall v\in H_{0}^{1}(\Omega)$:
\begin{equation}
(u^{k+1},v)+\alpha_{0}(\nabla u^{k+1},\nabla v)=((1-b_{1})u^{k}+\sum_{j=1}^{k-1}(b_{j}-b_{j+1})u^{k-j}+b_{k}u^{0}+\frac{1}{\mu}S^{k+1},v).\label{eq:weakform}
\end{equation}
Let $\mathbb{P}_{N}$ be the space of polynomials over $I$ with degree
not exceeding $N$ and $(\mathbb{P}_{N}^{0})^{2}=\{v\in(\mathbb{P}_{N})^{2}:\,v=0,\,\text{on}\,\partial\Omega\}$\@. 

The Galerkin formulation of the (\ref{eq:weakform}) is to find $u_{N}^{k+1}\in(\mathbb{P}_{N}^{0})^{2}$
such that $\forall v_{N}\in(\mathbb{P}_{N}^{0})^{2}$:
\begin{equation}
(u_{N}^{k+1},v_{N})+\alpha_{0}(\nabla u_{N}^{k+1},\nabla v_{N})=((1-b_{1})u_{N}^{k}+\sum_{j=1}^{k-1}(b_{j}-b_{j+1})u_{N}^{k-j}+b_{k}u_{N}^{0}+\frac{1}{\mu}I_{N}S^{k+1},v_{N}),\label{eq:Galerkin}
\end{equation}
with $\left(f,g\right)$ being the standard inner product and $I_{N}$
an interpolation operator\@. 

\subsubsection{Bernstein dual-Petrov-Galerkin method}

Since $\dim\mathbb{P}_{N}^{0}=N-1$, and due to (\ref{eq:BernBound}),
we choose a basis for it by removing the first and last Bernstein
polynomials of degree $N$, i.e.,
\begin{equation}
\mathbf{\Phi}  =  \left[\phi_{i}(x):\,1\leq i\leq N-1\right]^{T}.\label{eq:BernbasisVector}
\end{equation}
 Using (\ref{eq:BernDeriv}), it is easy to verify
\begin{equation}
  \mathbf{\Phi}'=\mathbf{D\Phi}+\mathbf{d},\label{eq:BernDerMatrix}
\end{equation}
where $\mathbf{D}=\mathrm{tridiag}(N-i+1,2i-N,-(i+1))$ is a tridiagonal
matrix of order  $N-1$ and $\mathbf{d}=N\left[\phi_{0},0,\dots,0,-\phi_{N}\right]^{T}$
is an $(N-1)-$vector.

Assuming $N_{x}=N_{y}=N,$ we use the tensor product of the basis
functions of $\mathbb{P}_{0}^{N}$ as a basis for two dimensional
case, $\left\{ \phi_{i}(x)\phi_{j}(y):1\leq i,j\leq N-1\right\} $
and consider an approximate solution of (\ref{eq: semi-discrete})
as 
\begin{equation}
 u_{N}^{k+1}=\sum_{i,j=1}^{N-1}{\hat{u}_{i,j}^{k+1}\phi_{i}(x)\phi_{j}(y)}=\mathbf{\Phi}^{T}\left(x\right)\mathbf{U}^{k+1}\mathbf{\Phi}\left(y\right),\quad(x,y)\in\Omega,\label{eq:ApproxSol}
\end{equation}
where $\mathbf{\Phi}\left(\cdot\right)=\left[\phi_{i}(\cdot):\,1\leq i\leq N-1\right]^{T}$
and $\mathbf{U}^{k+1}=[\hat{u}_{i,j}^{k+1}]$. Let us use the following
notations. 
\begin{align}
 & a_{i,j}=\int_{I}{\phi_{j}^{\prime}(x)\psi_{i}^{\prime}(x)dx},\quad\mathbf{A}=[a_{i,j}],\nonumber \\
 & b_{i,j}=\int_{I}{\phi_{j}(x)\psi_{i}(x)dx},\quad\mathbf{B}=[b_{i,j}],\nonumber \\
 & f_{i,j}^{k+1}=\int_{\Omega}{I_{N}f^{k+1}(x,y)\psi_{j}(x)\psi_{i}(y)d\Omega},\label{eq:integralsOFf}
\end{align}
for $1\leq j\leq N-1,\,0\leq i\leq N-2.$ 

Taking the test functions of (\ref{eq:Galerkin}) as $v=\psi_{l}(x)\psi_{m}(y)$
for $l,m=0,1,\dots,N-2,$ it is seen that the spectral form (\ref{eq:Galerkin})
is equivalent to the following linear system:
\begin{equation*}
 \mu_{\tau}^{\alpha}\mathbf{B}\mathbf{U}^{k+1}\mathbf{B}^{T}+\kappa\left(\mathbf{A}\mathbf{U}^{k+1}\mathbf{B}^{T}+\mathbf{B}\mathbf{U}^{k+1}\mathbf{A}^{T}\right)=\mathbf{F}^{k+1},\quad k\geq0,
\end{equation*}
that can be equivalently written as a Sylvester equation but it requires
computing the inverse of \textbf{$\mathbf{B}$}\@. Although $\mathbf{B}$
has only three nonzero diagonals, it can be shown that its inverse
is a dense matrix and so we avoid transforming to Sylvester equation.
Instead we use the equivalent tensor product form 
\begin{equation}
 \left(\mu_{\tau}^{\alpha}\mathbf{B}\otimes\mathbf{B}+\kappa\left(\mathbf{B}\otimes\mathbf{A}+\mathbf{A}\otimes\mathbf{B}\right)\right)\mathbf{u}^{k+1}=\mathbf{f}^{k+1},\label{eq:TensorProdSys}
\end{equation}
with $\mathbf{f=}\left[f_{1,0,\cdots},f_{q,0};f_{1,1\cdots},f_{q,1};\dots;f_{1,q-1,\cdots},f_{q,q-1}\right]^{T},\,q=N-1$\@.
It is worth to note that the coefficient matrix of the linear system
(\ref{eq:TensorProdSys}) is the same for all time steps so it is
to be evaluated just once for all $k\geq0.$

In terms of the trial vector (\ref{eq:BernbasisVector}), and test
vector (\ref{eq:ModalBasisVector}), we may write 
\begin{equation*}
  \mathbf{A=}\int_{I}{\mathbf{\Psi^{\prime}}\mathbf{\Phi}^{\prime T}dx},\mathbf{\quad B}=\int_{I}{\mathbf{\Psi}\mathbf{\Phi}^{T}}.
\end{equation*}
To facilitate the computations, in what follows, these matrices are
related to the transformation matrices introduced in Section \ref{subsec:Transformation-matrices-and}.
First, note that by the biorthogonality (\ref{eqBiorthogonality}),
we have
\begin{equation}
\int_{I}{\mathbf{\tilde{\Psi}}\mathbf{\Phi}^{T}dx}  =  \left[\begin{array}{ccccc}
0 & 0 & \cdots & 0 & 0\\
1 & 0 & \cdots & 0 & 0\\
0 & 1 & \cdots & 0 & 0\\
\vdots & \vdots & \ddots & \vdots & \vdots\\
0 & 0 & \cdots & 1 & 0\\
0 & 0 & \cdots & 0 & 1\\
0 & 0 & \cdots & 0 & 0
\end{array}\right]=:\mathbf{\tilde{I}}.\label{eq:Itilde}
\end{equation}
Now from (\ref{eq:SaiToSaiTilde(G)}), and writing $\mathbf{G}$ as
$\mathbf{G}=[\mathbf{g}_{0},\mathbf{g}_{1},\dots,\mathbf{g}_{N}]$,
we get 
\begin{equation*}
\mathbf{B}  =  \int_{I}{\mathbf{G}\mathbf{\tilde{\Psi}}\mathbf{\Phi}^{T}dx}=\mathbf{G}\mathbf{\tilde{I}}=[\mathbf{g}_{1},\mathbf{g}_{2},\dots,\mathbf{g}_{N-1}].
\end{equation*}
So $\mathbf{B}$ is a tridiagonal matrix whose entries are given by
\begin{align}
b_{i,j}  =  \begin{cases}
1, & j=i+1,\\
a_{i}, & j=i,\\
b_{i}, & j=i-1,\\
0, & otherwise,
\end{cases}\label{eq:Bentries}
\end{align}
where $a_{i}$'s and $b_{i}$'s are easily computed by (\ref{eq:a=00005Bi=00005D and b=00005Bi=00005D}).
On the other hand, from Lemma \ref{Lemma1forQ}, the Bernstein operational
matrix of differentiation (\ref{eq:BernDerMatrix}) and (\ref{eq:Itilde}),
we obtain
\begin{align}
\mathbf{A} & =  \int_{I}{\mathbf{Q}\tilde{\mathbf{\Psi}}(\mathbf{\Phi^{T}D^{T}}+\mathbf{d^{T}})dx}\nonumber \\
 & =  \mathbf{Q}\mathbf{\tilde{I}}\mathbf{D^{T}}+N[\mathbf{q}_{0},\mathbf{0},\dots,\mathbf{0},\mathbf{q}_{N}]\nonumber \\
 & =  [\mathbf{q}_{1},\dots,\mathbf{q}_{N-1}]\mathbf{D}^{T}+N[\mathbf{q}_{0},\mathbf{0},\dots,\mathbf{0},\mathbf{q}_{N}],\label{eq:Amatrix}
\end{align}
 where $\mathbf{Q}=[\mathbf{q}_{0},\mathbf{q}_{1},\dots,\mathbf{q}_{N}]$
is a $\mathrm{(1,3)-band}$ matrix introduced Lemma \ref{Lemma1forQ}.
Hence, $\mathbf{Q}\mathbf{\tilde{I}}$ is a pentadiagonal matrix and
$\mathbf{A}$ is the product of a pentadiagonal and a tridiagonal
matrix plus a sparse matrix. From Lemma \ref{Lemma1forQ} and (\ref{eq:Amatrix}),
it is seen that $\mathbf{A}$ is a seven-diagonal matrix. 

Notice that the solution of linear system (\ref{eq:TensorProdSys})
requires the matrices $\mathbf{A}$ and $\mathbf{B}.$ $\mathbf{A}$
is obtained by a sparse matrix-matrix multiplication (\ref{eq:Amatrix})
and entries of $\mathbf{B}$ are given by (\ref{eq:Bentries}). 
\begin{remark}
Since the coefficient matrix of the linear system (\ref{eq:TensorProdSys})
remains intact for a fixed $\tau$, only the RHS vector to be computed
for different time steps, $k=0,1,\dots$ up to a desired time. So
it is efficient to use a band-LU factorization for solving the system.
It is remarkable that for a $(2p+1)-$band matrix, the LU-factorization
can be done with $O(Np^{2})$ flops and backward substitutions require
$O(Np)$ flops {[}\citealp{Golub}, Section 4.3{]}.
\end{remark}

\section{\label{sec:Error-estimation}Stability and convergence analysis}

For the error analysis, we assume the problem (\ref{eq:main}) to
be homogeneous, $S=0$. 

For $\alpha\geq0,$ the bilinear form $a\left(u,v\right)=\left(\nabla u,\nabla v\right)+\alpha(u,v)$
in (\ref{eq:Galerkin}) is continuous and coercive in $H_{0}^{1}\left(\Omega\right)\times H_{0}^{1}\left(\Omega\right)$.
The existence and uniqueness of the solution for both the weak form
(\ref{eq:Galerkin}) and the Galerkin form (\ref{eq:Galerkin}) is
guarantied by the well-known Lax-Milgram lemma.

We define the following inner product and the associated energy norm
on $H_{0}^{1}(\Omega)$: 
\begin{equation}
  (u,v)=\int_{\Omega}uvd\Omega,\quad(u,v)_{1,}=(u,v)+\alpha_{0}(\nabla u,\nabla v),\quad\|u\|_{1}=(u,u)_{1}^{\frac{1}{2}}.\label{eq:EnergyNorm}
\end{equation}

\begin{theorem}
\label{Th:stabilityForThe-weak-semidiscrete}The weak form (\ref{eq:weakform})
is unconditionally stable: 
\begin{equation}
  \|u^{k}\|_{1}\leq\|u^{0}\|,\quad k=1,\dots,M.\label{eq:stabilityOfSemiDiscrete}
\end{equation}
\end{theorem}
\begin{proof}
Let $v=u^{1}$ in (\ref{eq:weakform}). Then, 
\begin{equation*}
  (u^{1},u^{1})+\alpha_{0}(\nabla u^{1},\nabla u^{1})=(u^{0},u^{1}),
\end{equation*}
giving (\ref{eq:stabilityOfSemiDiscrete}) for $k=1,$ by  the definition
(\ref{eq:EnergyNorm}), the Schwarz inequality and the inequality
$\|v\|\leq\|v\|_{1}$. By mathematical induction, assume (\ref{eq:stabilityOfSemiDiscrete})
holds for $k=0,\dots,n.$ Let $v=u^{n+1}$ in (\ref{eq:weakform}),
i.e.,
\begin{align*}
 & (u^{n+1},u^{n+1})+\alpha_{0}(\nabla u^{n+1},\nabla u^{n+1})=(1-b_{1})(u^{n},u^{n+1})\\
 &   \quad\quad\quad\hfill+\sum_{j=1}^{n-1}(b_{j}-b_{j+1})(u^{n-j},u^{n+1})+b_{n}(u^{0},u^{n+1}).
\end{align*}
It is easy to see that the RHS coefficients in (\ref{eq: semi-discrete})
are positive. So we obtain 
\begin{align*}
\|u^{n+1}\|_{1} & \leq  (1-b_{1})\|u^{n}\|+\sum_{j=1}^{n-1}(b_{j}-b_{j+1})\|u^{n-j}\|+b_{n}\|u^{0}\|\\
 & \leq  \left((1-b_{1})+\sum_{j=1}^{n-1}(b_{j}-b_{j+1})+b_{n}\right)\|u^{0}\|=\|u^{0}\|.
\end{align*}
So the proof is done. 
\end{proof}

\begin{theorem}
\label{Th:ConvergenceForSemidiscrete}Let $u$ be the solution of
the equation (\ref{eq:main}) with conditions (\ref{IV})-(\ref{BVs})
and $u^{k}$ be the solution of the the semidiscrete problem (\ref{eq: semi-discrete}).
Then,
\begin{align}
 &   \|u(t_{k})-u^{k}\|_{1}\leq\frac{c_{u}}{1-\alpha}T^{\alpha}\tau^{2-\alpha},\quad0<\alpha<1,\label{eq:Error u(t_k)-u^k}\\
 &   \|u(t_{k})-u^{k}\|_{1}\leq c_{u}T\tau,\quad\qquad\quad\,as\,\alpha\rightarrow1.\label{eq:Error u-u^k alpha->1}
\end{align}
\end{theorem}
\begin{proof}
The idea of the proof comes from \citep{Lin}. We first prove 
\begin{equation}
    \|u(t_{k})-u^{k}\|_{1}\leq\frac{c_{u}}{b_{k-1}}\tau^{2},\quad k=1,\dots,M.\label{eq:komaki}
\end{equation}
By (\ref{eq:main}) and (\ref{eq:semi-discFork=00003D0}), we have
\begin{equation*}
    (e^{1},v)+\alpha_{0}(\nabla e^{1},\nabla v)=(e^{0},v)+(r^{1},v),\quad\forall v\in H_{0}^{1}(\Omega),
\end{equation*}
in which $e^{k}:=u(t_{k})-u^{k}.$ For $v=e^{1}$ and by using $e^{0}=0$,
$\|v\|\leq\|v\|_{1}$ and (\ref{eq:TruncationErrorMulti}), we get
\begin{equation}
  \|e^{1}\|_{1}\leq c_{u}\tau^{2},\label{eq:Forj=00003D1}
\end{equation}
i.e., (\ref{eq:komaki}) holds for $k=1$. By induction, assume (\ref{eq:komaki})
holds for $k\leq n$. Using (\ref{eq:main}) and (\ref{eq: semi-discrete}),
we get
\begin{align*}
 &   (e^{n+1},v)+\alpha_{1}(\nabla e^{n+1},\nabla v)=(1-b_{1})(e^{n},v)\\
 &   \quad\quad\quad+\sum_{j=1}^{n-1}(b_{j}-b_{j+1})(e^{n-j},v)+b_{n}(e^{0},v)+(r^{n+1},v),\quad\forall v\in H_{0}^{1}(\Omega).
\end{align*}
For $v=e^{n+1}$, it reads as
\begin{align*}
\|e^{n+1}\|_{1}^{2} & \leq  (1-b_{1})\|e^{n}\|\|e^{n+1}\|_{1}+\sum_{j=1}^{n-1}(b_{j}-b_{j+1})\|e^{n-j}\|\|e^{n+1}\|_{1}+\|r^{n+1}\|\|e^{n+1}\|_{1},\\
\Rightarrow\|e^{n+1}\|_{1} & \leq  (1-b_{1})\frac{c_{u}}{b_{n-1}}\tau^{2}+\sum_{j=1}^{n-1}(b_{j}-b_{j+1})\frac{c_{u}}{b_{n-j-1}}\tau^{2}+c_{u}\tau^{2}\\
 & \leq  \left((1-b_{1})+\sum_{j=1}^{n-1}(b_{j}-b_{j+1})+b_{n}\right)\frac{c_{u}}{b_{n}}\tau^{2}=\frac{c_{u}}{b_{n}}\tau^{2},
\end{align*}
proving (\ref{eq:komaki}) for $k=n+1$ that completes the proof of
(\ref{eq:komaki}). 

Consider $f(t)=t^{1-\alpha}$, then there exists a$\xi,$ $k-1<\xi<k\leq M$
such that
\begin{equation*}
  b_{k-1}\tau^{-\alpha}=\frac{(k\tau)^{1-\alpha}-(\tau(k-1))^{1-\alpha}}{\tau}=(1-\alpha)(\xi\tau)^{-\alpha}\geq(1-\alpha)(k\tau)^{-\alpha}\geq(1-\alpha)(T)^{-\alpha},
\end{equation*}
which gives
\begin{equation}
  \frac{c_{u}}{b_{k-1}}\tau^{2}\leq\frac{c_{u}}{1-\alpha}T^{\alpha}\tau^{2-\alpha}.\label{eq:boundFor b_k}
\end{equation}
Now using this along with (\ref{eq:komaki}) proves (\ref{eq:Error u(t_k)-u^k}). 

In order to derive (\ref{eq:Error u-u^k alpha->1}), we first prove
\begin{equation}
  \|u(t_{k})-u^{k}\|_{1}\leq c_{u}k\tau^{2},\quad k=1,\dots,M.\label{eq:komaki2}
\end{equation}
 By (\ref{eq:Forj=00003D1}), the inequality (\ref{eq:komaki2}) holds
for $k=1$. Assume (\ref{eq:komaki2}) holds for $k=1,\dots,n,\:n\leq M-1$.
Then, from (\ref{eq:main}), (\ref{eq: semi-discrete}) and (\ref{eq:TruncationErrorMulti}),
we obtain
\begin{align*}
\|e^{n+1}\|_{1} & \leq  (1-b_{1})\|e^{n}\|+\sum_{j=1}^{n-1}(b_{j}-b_{j+1})\|e^{n-j}\|+\|r^{n+1}\|\\
 & \leq  \left((1-b_{1})\frac{n}{n+1}+\sum_{j=1}^{n-1}(b_{j}-b_{j+1})\frac{n-j}{n+1}+\frac{1}{(n+1)}\right)c_{u}(n+1)\tau^{2}\\
 & \leq  \left((1-b_{1})\frac{n}{n+1}+(b_{1}-b_{n})\frac{n}{n+1}-(b_{1}-b_{n})\frac{1}{n+1}+\frac{1}{(n+1)}\right)c_{u}(n+1)\tau^{2}\\
 & =  \left(1-b_{n}\frac{n}{n+1}-(b_{1}-b_{n})\frac{1}{n+1}\right)c_{u}(n+1)\tau^{2}\leq c_{u}(n+1)\tau^{2}.
\end{align*}
So (\ref{eq:komaki2}) holds for $k=n+1$. From $k\tau\leq T$ and
(\ref{eq:komaki2}), we get (\ref{eq:Error u-u^k alpha->1}).
\end{proof}

\subsection{Convergence of the full discretization scheme}

Let $\pi_{N}^{1,0}$ be the $H^{1}$-orthogonal projection operator
from $H_{0}^{1}(\Omega)$ into $(\mathbb{P}_{N}^{0})^{2}$ associated
with the energy norm $\|\cdot\|_{1}$ defined in (\ref{eq:EnergyNorm}).
Due to the equivalence of this norm with the standard $H^{1}$ norm,
we have the following error estimation {[}\citealp{Lin}; Relation
(4.3){]}
\begin{equation}
  \|u-\pi_{N}^{1,0}u\|_{1}\leq cN{}^{1-m}\|u\|_{m},\quad u\in H_{0}^{m}(\Omega)\cap H_{0}^{1}(\Omega),\,m\geq1.\label{eq:projectionError}
\end{equation}

The idea of the proof for the following result comes from the paper
\citep{Lin}.
\begin{theorem}
Let $u^{k},k=0,\dots,M$ be the solution of the variational formulation
(\ref{eq:weakform}) and $u_{N}^{k}$ be the solution of the scheme
(\ref{eq:Galerkin}), assuming $u^{0}=\pi_{N}^{1,0}u^{0}$ and $u^{k}\in H^{m}(\Omega)\cap H_{0}^{1}(\Omega)$
for some $m>1.$ Then,
\begin{align}
 & \|u^{k}-u_{N}^{k}\|_{1}\leq\frac{c}{1-\alpha}\tau^{-\alpha}N{}^{1-m}\max_{0\leq j\leq k}\|u^{j}\|_{m},\quad0<\alpha<1,\nonumber \\
 &  \|u^{k}-u_{N}^{k}\|_{1}\leq cN^{1-m}\sum_{j=0}^{k}\|u^{j}\|_{m},\hfill\qquad\,\quad\alpha\rightarrow1,\label{eq:Error u^k-(u_N)^k}
\end{align}
for $k=1,\dots,M$, where $c$ depends only on $T^{\alpha}$.
\end{theorem}
\begin{proof}
We have $(u^{k+1}-\pi_{N}^{1,0}u^{k+1},v_{N})_{1}=0,\,\forall v_{N}\in(\mathbb{P}_{N}^{0})^{2}$
by the projection operator. By definition of the norm (\ref{eq:EnergyNorm}),
we get 
\begin{equation*}
  (\pi_{N}^{1,0}u^{k+1},v_{N})+\alpha_{1}(\nabla\pi_{N}^{1,0}u^{k+1},\nabla v_{N})=(u^{k+1},v_{N})+\alpha_{1}(\nabla u^{k+1},\nabla v_{N}),\quad\forall v_{N}\in(\mathbb{P}_{N}^{0})^{2}.
\end{equation*}
By the weak form (\ref{eq:weakform}), the RHS of the above equation
is replaced as 
\begin{align}
 &   (\pi_{N}^{1,0}u^{k+1},v_{N})+\alpha_{1}(\nabla\pi_{N}^{1,0}u^{k+1},\nabla v_{N})=(1-b_{1})(u^{k},v_{N})\nonumber \\
 &   \quad\quad\quad+\sum_{j=1}^{k-1}(b_{j}-b_{j+1})(u^{k-j},v_{N})+b_{k}(u^{0},v_{N}),\quad\forall v_{N}\in(\mathbb{P}_{N}^{0})^{2}.\label{eq:11}
\end{align}
Subtracting (\ref{eq:11}) from (\ref{eq:Galerkin}), we have
\begin{align*}
 &   (\tilde{e}_{N}^{k+1},v_{N})+\alpha_{1}(\frac{\partial\tilde{e}_{N}^{k+1}}{\partial x},\frac{\partial v_{N}}{\partial x})=(1-b_{1})(e_{N}^{k},v_{N})\\
 &   \quad\quad\quad+\sum_{j=1}^{k-1}(b_{j}-b_{j+1})(e_{N}^{k-j},v_{N})+b_{k}(e_{N}^{0},v_{N}),\quad\forall v_{N}\in(\mathbb{P}_{N}^{0})^{2},
\end{align*}
where $e_{N}^{k+1}=u^{k+1}-u_{N}^{k+1}$ and $\tilde{e}_{N}^{k+1}=\pi_{N}^{1,0}u^{k+1}-u_{N}^{k+1}$.
Let $v_{N}=\tilde{e}_{N}^{k+1}$, then
\begin{equation*}
  \|\tilde{e}_{N}^{k+1}\|_{1}\leq(1-b_{1})\|e_{N}^{k}\|+\sum_{j=1}^{k-1}(b_{j}-b_{j+1})\|e_{N}^{k-j}\|+b_{k}\|e_{N}^{0}\|.
\end{equation*}
With $\|e_{N}^{k+1}\|_{1}\leq\|\tilde{e}_{N}^{k+1}\|_{1}+\|u^{k+1}-\pi_{N}^{1,0}u^{k+1}\|_{1}$,
we obtain
\begin{equation*}
  \|e_{N}^{k+1}\|_{1}\leq(1-b_{1})\|e_{N}^{k}\|+\sum_{j=1}^{k-1}(b_{j}-b_{j+1})\|e_{N}^{k-j}\|+b_{k}\|e_{N}^{0}\|+cN^{1-m}\|u^{k+1}\|.
\end{equation*}
As in the proof of Theorem \ref{Th:ConvergenceForSemidiscrete}, it
is first proved by induction that:
\begin{align*}
 &   \|e_{N}^{k+1}\|_{1}\leq\frac{1}{b_{k-1}}\max_{0\leq j\leq k}\|u^{j}-\pi_{N}^{1,0}u^{j}\|_{1},\quad0<\alpha<1,\\
 &   \|e_{N}^{k+1}\|_{1}\leq\sum_{j=0}^{k}\|u^{j}-\pi_{N}^{1,0}u^{j}\|_{1},\quad\alpha\rightarrow1,
\end{align*}
 for $0\leq k\leq M.$ Then, by using (\ref{eq:boundFor b_k}) and
the projection error (\ref{eq:projectionError}) the desired result
is derived. 
\end{proof}

The following theorem is obtained by the triangle inequality $||u(\cdot,t_{k})-u_{N}^{k}||_{1}\leq||u(\cdot,t_{k})-u^{k}||_{1}+||u^{k}-u_{N}^{k}||_{1}$
along with the inequalities (\ref{eq:Error u(t_k)-u^k}) and (\ref{eq:Error u^k-(u_N)^k}).
\begin{theorem}
Let $u$ be the solution of the problem (\ref{eq:main}) with the
initial and boundary conditions given by (\ref{IV})-(\ref{BVs})
and $u_{N}^{k}$ be the solution of the scheme (\ref{eq:Galerkin}).
Then, assuming $u_{N}^{0}=\pi_{N}^{1,0}u^{0}$ and $u\in H^{m}(\Omega)\cap H_{0,}^{1}(\Omega)$,
we have
\begin{align}
 &   \|u(t_{k})-u_{N}^{k}\|_{1}\leq\frac{CT^{\alpha}}{1-\alpha}(c_{u}\tau^{2-\alpha}+c\tau^{-\alpha}N{}^{1-m}\sup_{0<t<T}\|u(x,t)\|_{m}),\quad k\leq M,\quad0<\alpha<1,\label{eq:SpatialRate}\\
 &   \|u(t_{k})-u_{N}^{k}\|_{1}\leq T^{\alpha}(c_{u}\tau+c\tau^{-1}N{}^{1-m}\sup_{0<t<T}\|u(x,t)\|_{m}),\quad k\leq M,\quad\alpha\rightarrow1.\nonumber 
\end{align}
The constants $C$ and $c$ are independent of $\tau$, $T$, $N$.
\end{theorem}
It is seen that the method has the so-called spectral convergence
in space and the order of convergence $O(\tau^{2-\alpha})$ in time.

\section{\label{sec:Numerical-examples}Numerical examples}

Here, some numerical experiments are provided to show the accuracy
of the proposed method. For the computations, we use Maple 18. on
a laptop with CPU core i3 1.9 GHz and RAM 4 running Windows 8.1 platform.
To compute the errors, we use the discrete $L^{2}$ and $L^{\infty}$
errors defined as
\begin{align*}
 &   L^{2}\approx\left(\frac{1}{\mathcal{N}^{2}}\sum_{i,j=0}^{\mathcal{N}-1}{|u(x_{i},y_{j},t_{m})-u_{N}^{m}(x_{i},y_{j})|^{2}}\right)^{1/2},\\
 &   L^{\infty}\approx\max_{0\leq i,j\leq\mathcal{N}}{|u(x_{i},y_{j},t_{m})-u_{N}^{m}(x_{i},y_{j})|},
\end{align*}
respectively, where $u$ is the exact solution of the problem (\ref{eq:main})-(\ref{BVs}),
$u_{N}^{m}$ is the approximation solution (\ref{eq:ApproxSol}) at
$t=t_{m}=m\tau,$ $x_{i}=\frac{i}{\mathcal{N}},\,y_{j}=\frac{j}{N}$
and $\mathcal{N}=100$. Also, the convergence rates in space and time
are respectively computed by
\begin{equation*}
  \text{rate}_{N_{i}}=\frac{\log{\frac{E(N_{i},\tau)}{E(N_{i-1},\tau)}}}{\log{\frac{N_{i-1}}{N_{i}}}},\quad\text{rate}_{\tau_{i}}=\frac{\log{\frac{E(N,\tau_{i})}{E(N,\tau_{i-1})}}}{\log{\frac{\tau_{i}}{\tau_{i-1}}}},
\end{equation*}
where $E(h,\tau)$ is the error with $h=1/N$ where $N$ stands for
the dimension of basis and $\tau$ is the time-step size. However,
as it is common in the literature, we will show the spectral convergence
of the proposed method by logarithmic scaled error plots.

It is worth to mention that as we derived the operational matrices
in Section \ref{subsec:Transformation-matrices-and} with special
structures, the proposed method finally leads to the linear system
(\ref{eq:TensorProdSys}) that is sparse and banded. To see the sparsity
and bandedness, the nonzero entries of the coefficinet matrix of (\ref{eq:TensorProdSys})
are depicted in Fig. \ref{fig:The-coefficient-matrix} using \textit{sparsematrixplot}
command of Maple. It is seen that the density decreases rapidly as
$N$ increases.

\begin{figure}
\centering
\includegraphics[scale=0.4]{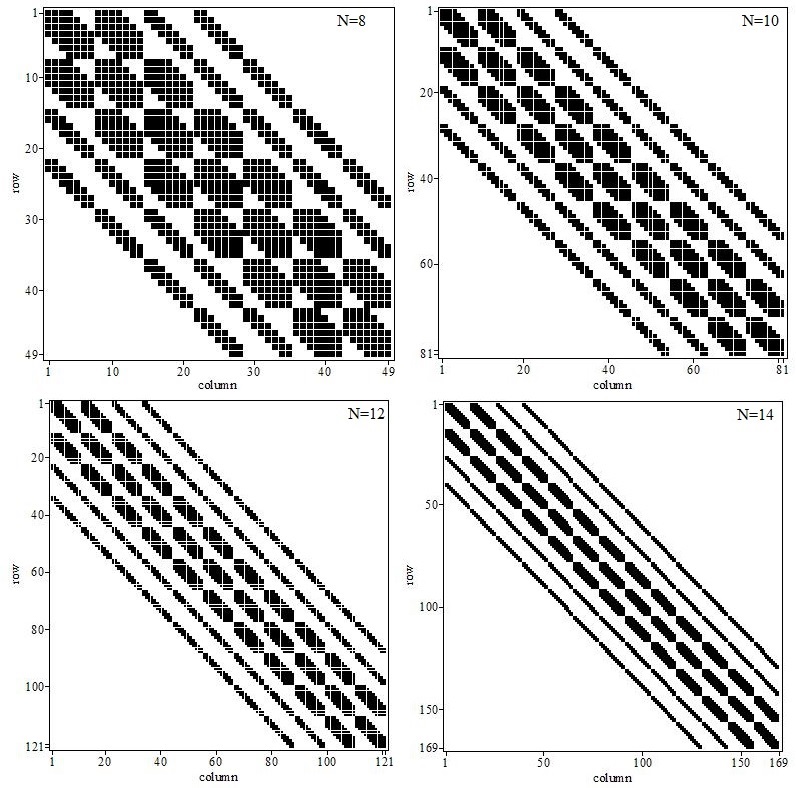}
\par
\centering\caption{\label{fig:The-coefficient-matrix}The coefficient matrix of the linear
system (\ref{eq:TensorProdSys}).}
\end{figure}

\begin{example}
\label{exa:1 sin(pi*x)sin(pi*y)t^2}Consider the problem (\ref{eq:main})
with $\kappa=1$ and the exact solution $u(x,y,t)=\sin{(\pi x)}\sin{(\pi y)}t^{2}$\@.
Table \ref{tab:SpatialRateExample1} shows the convergence of the
method in space for $\tau=\Delta t=0.01$ for fractional orders $\alpha=0.25,0.50,0.75$.
$N_{x}$ and $N_{y}$ stand for the number of basis in $x$ and $y$
direction. Figure \ref{fig:ex1} demonstrates the logarithmic scale
error plot in terms of $H^{1}$-norm for $\alpha=1/2$ for some $t$'s.
It is seen that the method preserves the spectral convergence at different
time rows $t<1$ and $t>1$. Table \ref{tab:timeRate} reports the
convergence in time by considering $N_{x}=N_{y}=N=8$ as $\tau$ decreases
at time rows $t=0.1$ and $t=1$. It verifies the $O(\tau^{2-\alpha})$
temporal rate of convergence. 
\end{example}
\begin{figure}
\centering
\includegraphics[scale=0.7]{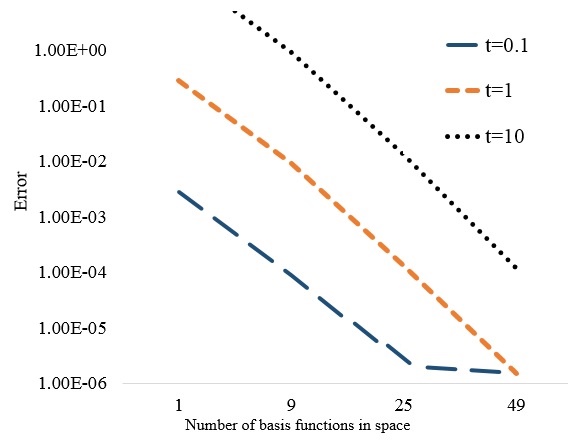}
\centering\caption{\label{fig:ex1}Convergence of the spectral method (\ref{eq:TensorProdSys})
in space with $\tau=0.001$.}
\end{figure}

\begin{table}[h]
\centering\caption{\label{tab:SpatialRateExample1} Convergence in space at $t=1$ for
Example \ref{exa:1 sin(pi*x)sin(pi*y)t^2}.}

\centering{}%
\begin{tabular}{ccccccc}
\hline 
 & \multicolumn{2}{c}{$\alpha=0.25$} & \multicolumn{2}{c}{$\alpha=0.5$} & \multicolumn{2}{c}{$\alpha=0.75$}\tabularnewline
\hline 
$N_{x}=N_{y}$ & $L^{\infty}$ & $H^{1}$ & $L^{\infty}$ & $H^{1}$ & $L^{\infty}$ & $H^{1}$\tabularnewline
\hline 
2 & 7.53E-02 & 2.81E-01 & 7.52E-02 & 2.81E-01 & 7.49E-02 & 2.81E-01\tabularnewline
4 & 1.74E-03 & 8.91E-03 & 1.72E-03 & 8.91E-03 & 1.63E-03 & 8.91E-03\tabularnewline
6 & 1.78E-05 & 1.34E-04 & 3.01E-05 & 1.43E-04 & 9.98E-05 & 2.86E-04\tabularnewline
8 & 3.67E-06 & 8.75E-06 & 2.25E-05 & 5.15E-05 & 2.79E-06 & 2.54E-04\tabularnewline
\hline 
\end{tabular}
\end{table}

\begin{example}
\label{exa:ex2NoSourse}To see the method works for the case in which
there is no source term, consider the problem (\ref{eq:main})-(\ref{BVs})
with the initial condition $u(x,y,0)=x(x-1)\sin{(2\pi y)}$, $\kappa=1$
and no source term \citep{Yang2014}. 

The spectral convergence in space is seen from Fig. \ref{fig:ex2 conv in space}
in which the time step length is considered to be $\tau=0.01.$ The
solution with $N_{x}=N_{y}=10$ is treated as the exact solution.
The errors are reported at $t=1$ with $H^{1}$-norm. 

The numerical examples present the spectral convergence in space and
fixed convergence of $O(\tau^{2-\alpha})$ in time confirming the
theoretical claims. It should be noted that we have used the eight
point Gauss-Legendre quadrature rule to perform the integrals (\ref{eq:integralsOFf})
in the right hand side of the linear system (\ref{eq:TensorProdSys}).
\end{example}
\begin{table}
\centering\caption{\label{tab:timeRate}Error and temporal rate of convergence at $t=1$
for Example \ref{exa:1 sin(pi*x)sin(pi*y)t^2}. }

\noindent\resizebox{\textwidth}{!}{%
\centering{}%
\begin{tabular}{ccccccccccccc}
\hline 
 & \multicolumn{4}{c}{$\alpha=0.25$} & \multicolumn{4}{c}{$\alpha=0.5$} & \multicolumn{4}{c}{$\alpha=0.75$}\tabularnewline
\hline 
 & \multicolumn{2}{c}{$t=0.1$} & \multicolumn{2}{c}{$t=1$} & \multicolumn{2}{c}{$t=0.1$} & \multicolumn{2}{c}{$t=1$} & \multicolumn{2}{c}{$t=0.1$} & \multicolumn{2}{c}{$t=1$}\tabularnewline
\hline 
$M$ & $H^{1}$ & rate & $H^{1}$ & rate & $H^{1}$ & rate & $H^{1}$ & rate & $H^{1}$ & rate & $H^{1}$ & rate\tabularnewline
\hline 
10 & 2.90E-04 &  & 4.21E-04 &  & 1.16E-03 &  & 1.55E-03 &  & 3.27E-03 &  & 4.46E-03 & \tabularnewline
20 & 9.93E-05 & 1.55 & 1.32E-04 & 1.67 & 4.60E-04 & 1.33 & 5.61E-04 & 1.47 & 1.54E-03 & 1.09 & 1.89E-03 & 1.24\tabularnewline
40 & 3.27E-05 & 1.60 & 4.12E-05 & 1.68 & 1.73E-04 & 1.41 & 2.01E-04 & 1.48 & 6.87E-04 & 1.16 & 7.95E-04 & 1.25\tabularnewline
80 & 1.05E-05 & 1.64 & 1.27E-05 & 1.70 & 6.35E-05 & 1.45 & 7.18E-05 & 1.49 & 2.97E-04 & 1.21 & 3.35E-04 & 1.25\tabularnewline
160 & 3.31E-06 & 1.67 & 4.05E-06 & 1.65 & 2.30E-05 & 1.47 & 2.56E-05 & 1.49 & 1.26E-04 & 1.24 & 1.41E-04 & 1.25\tabularnewline
\hline 
\end{tabular}}
\end{table}

\begin{figure}
\centering
\includegraphics[scale=0.7]{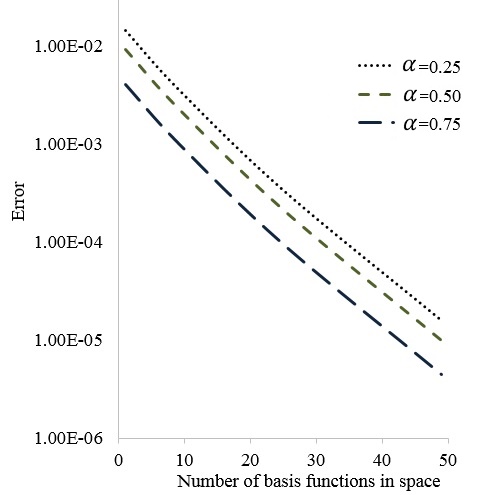}
\centering\caption{\label{fig:ex2 conv in space}Convergence in space for some fractional
orders with $\tau=0.01$.}
\end{figure}

\section{\label{sec:Con}Conclusion}

In this paper, some new aspects of the dual Bernstein polynomials
have been discussed. A suitable compact combinations of these polynomials
has been derived for developing a dual-Petrov-Galerkin variational
formulation for the numerical simulation of the two-dimensional subdiffusion
equation. It was shown that the method leads to sparse linear systems.
The illustrated numerical examples have been provided to show the
accuracy of the method. It is important to note that the transformation
matrices and the operational matrix for differentiation of dual Bernstein
polynomials that have been obtained in this work can be used similarly
for developing Bernstein-based dual-Petrov-Galerkin Galerkin methods
for other fractional partial differential equations on bounded domains.

\end{document}